\documentclass[reqno]{amsart}
\usepackage{amsmath}
\usepackage{amssymb}
\usepackage{amsthm}
\usepackage{latexsym}
\usepackage{hyperref}
\usepackage{enumerate}
\usepackage{graphicx}
\usepackage[all]{xy}
\usepackage{tikz-cd} 
\usepackage{color}
\usepackage{verbatim}
\usepackage{mathtools}

\usepackage[margin=1in]{geometry}
\usepackage[mathscr]{euscript}

\newcommand \fk[1]{{{\mathfrak #1}}}
\newcommand \C[1]{{\mathcal #1}}

\newcommand\fg{\mathfrak g}
\newcommand\fn{\mathfrak n}
\newcommand\fp{\mathfrak p}
\newcommand\fm{\mathfrak m}

\newcommand \bC{{\mathbb C}}
\newcommand \bF{{\mathbb F}}

\newcommand \bR{{\mathbb R}}
\newcommand \bZ{{\mathbb Z}}

\newcommand \bO{{\mathbb O}}

\newcommand\CA{{\C A}}

\newcommand\CF{{\C F}}
\newcommand\CG{{\C G}}
\newcommand\CH{{\C H}}

\newcommand\CO{{\C O}}

\newcommand\ft{{\mathfrak t}}

\newcommand\AZ{{\mathsf{AZ}}}

\newtheorem{theorem}{Theorem}[section]

\newtheorem{conjecture}[theorem]{Conjecture}

\newtheorem{lemma}[theorem]{Lemma}
\newtheorem{proposition}[theorem]{Proposition}
\newtheorem{definition}[theorem]{Definition}

\newtheorem{example}[theorem]{Example}

\newcommand\Inn{\mathsf{Inn}}

\newcommand\Fr{\mathsf{Fr}}

\newcommand\ad{\operatorname{ad}}

\newcommand\Ad{\operatorname{Ad}}

\def\<{\langle} 
\def\>{\rangle}

\numberwithin{equation}{section}

\begin{document}

\title[Unitary representations]{Unitary representations attached to parabolic subgroups: the case of abelian unipotent radical}

\author
{Dan Ciubotaru}
        \address[D. Ciubotaru]{Mathematical Institute, University of Oxford, Oxford OX2 6GG, UK}
        \email{dan.ciubotaru@maths.ox.ac.uk}

\begin{abstract} We classify the unitary representations with integral infinitesimal character in Lusztig's category of unipotent representations in the case when the geometric parameter space comes from the action of a Levi subgroup on the abelian nilradical of a (maximal) parabolic subalgebra. We organise the unitary representations into microlocal Arthur packets. This is a test case for investigating a conjectural description of unitary representations with integral infinitesimal character.
\end{abstract}


\maketitle

\section{Introduction}

Let $F$ be a nonarchimedean local field with residue field $\bF_q$, Weil group $W_F$, Frobenius element $\Fr$, and inertia group $I_F$. Let $\CG$ be a complex almost simple $F$-split group with connected centre, and $\Inn\CG$ the set of (pure) inner twists of $\CG$. Let $G$ be the complex Langlands dual group. A simplified \emph{Langlands parameter} is a continuous group homomorphism
\[
\varphi: W_F\times SL(2,\bC)\to G,
\]
satisfying certain properties. We will be interested in the parameters for which $I_F\subseteq \ker\varphi$. To each such $\varphi$, Lusztig \cite{Lu-unip} attached a finite packet $\Pi(\varphi)$ consisting of irreducible representations for the groups $\CG'(F)$, $\CG'\in \Inn \CG$. The set of packets $\Pi(\varphi)$, when $\varphi$ varies over $G$-conjugacy classes of Langlands parameters give a partition of the set of (equivalence classes) of irreducible representations for the groups $\CG'(F)$ with unipotent supercuspidal support.

An \emph{Arthur parameter} is a continuous group homomorphism
\[
\psi: W_F\times SL(2,\bC)\times SL(2,\bC)\to G.
\]
such that $\psi(W_F)$ is bounded. 
Conjecturally, to each $\psi$ there should exist a finite packet $\CA(\psi)$ consisting of irreducible unitary $\CG'(F)$-representations  satisfying a list of desiderata. The easiest expectations to state are:
\begin{enumerate}
    \item $\Pi(\varphi_\psi)\subseteq\CA(\psi)$, where $\varphi_\psi:W_F\times SL(2,\bC)\to G$ is defined by
    \[
    \varphi_\psi(w,1)=\psi(w,\left(\begin{matrix}||w||^{1/2}&0\\0&||w||^{-1/2}\end{matrix}\right),\left(\begin{matrix}||w||^{1/2}&0\\0&||w||^{-1/2}\end{matrix}\right))\text{ and }\varphi_\psi(1,\left(\begin{matrix}1&1\\0&1\end{matrix}\right))=\psi(1,\left(\begin{matrix}1&1\\0&1\end{matrix}\right),1);
    \]
    \item If $\psi^\vee$ denotes the parameter obtained from $\psi$ by flipping the two $SL(2,\bC)$, then \[\CA(\psi^\vee)=\AZ(\CA(\psi)),\] where $\AZ$ is the Aubert-Zelevinsky involution \cite{Au-inv}.
\end{enumerate}

When $\CG$ is a classical group, an endoscopic construction of the packets $\CA(\psi)$ was realised by Arthur \cite{Ar}. In general, the only uniform candidate for constructing $\CA(\psi)$ in the $p$-adic setting is the microlocal packet proposed by Adams-Barbasch-Vogan and Vogan \cite{ABV,Vo-llc}. See also \cite{Cun-et-al} for more details.

We restrict to the set of parameters $\psi$ with $I_F\subseteq \ker\psi$. Let $s=\varphi_\psi((\Fr,\left(\begin{matrix}q^{1/2}&0\\0&q^{-1/2}\end{matrix}\right)))$ be a semisimple element in $G$. We call (the $G$-conjugacy class of) $s$ the \emph{infinitesimal character} of $\psi$. We define in the same way the infinitesimal character of any Langlands parameter $\varphi$ with $I_F\subseteq \ker\varphi$. 

Let $\fg_q=\{x\in\fg\mid \Ad(s)x=qx\}$. The connected reductive group $G(s)$ acts with finitely many orbits on $V=\fg_q$. Lusztig's results imply that there is a one-to-one correspondence between the representations in the union of L-packets $\Pi(\varphi)$ such that $\varphi(\Fr)=s$ and the set of simple $G(s)$-equivariant perverse sheaves on the vector space $V$. If $\C E$ is an equivariant perverse sheaf, let us denote by $\pi(\C E)$ the corresponding irreducible representation in an L-packet $\Pi(\varphi)$. 

A fundamental invariant of $\C E$ is the \emph{singular support} or \emph{characteristic cycle}, which  is a $\bZ$-linear combination of closures of conormal bundles $\overline {T^*_\CO V}$,  $\CO$ ranging over the (finite) set of $G(s)$-orbits on $V$,
\[
CC(\C E)=\sum_{\CO} m_\CO(\C E) ~\overline {T^*_\CO V}.
\]
Following \cite{ABV}, define for every orbit $\CO$, the microlocal packet
\begin{equation*}
    \CA(\CO)=\{\C E\mid m_\CO(\C E)\neq 0\}
\end{equation*}
The \emph{microlocal Arthur packet} is
\begin{equation}
  \CA(\psi)=\{\pi(\C E)\mid \C E\in \CA(\CO_{\phi_\psi})\},  
\end{equation}
where $\CO_{\phi_\psi}$ is the $G(s)$-orbit of the nilpotent element $\varphi_\psi(\left(\begin{matrix}1&1\\0&1\end{matrix}\right))$.

\

Now suppose $s$ is \emph{integral}, meaning $s=q^\lambda$, where $\lambda$ is integral as in Definition \ref{d:integral}. Motivated by the results of Atobe and Minguez for symplectic and orthogonal groups \cite{AM}, also \cite{Liu-et-al} for further examples and conjectures, we are interested in the classification of unitary representations in the L-packets with integral $s$. A natural expectation given the results in {\it loc. cit.} is

\medskip

\noindent{\it Integral unitary representations}: 
    Suppose $\varphi$ is a Langlands parameter with $I_F\subseteq\ker\varphi$ and $\varphi(\Fr)$ integral. For every $\pi\in \Pi(\varphi)$, $\pi$ is unitary if and only if $\pi$ belongs to a microlocal packet $\CA(\psi)$. 

\medskip

We can say a bit more about the candidates for the microlocal packets $\CA(\psi)$ if $\psi$ has integral infinitesimal character $s=q^\lambda$. Fix a Cartan subalgebra $\ft$ of $\fg$ and assume without loss of generality that $\lambda\in \ft$ is dominant with respect to the roots $\Phi$ of $\ft$ in $\fg$. Since we are interested in unitary representations of $p$-adic groups, we may restrict to those $\lambda$ that belong to the \emph{fundamental parallelipiped} (FPP)
\begin{equation}
    0\le \langle\lambda,\alpha\rangle\le 1,\text{ for all }\alpha\in\Phi\text{ simple root}.
\end{equation}
The FFP bound comes from a conjecture of Vogan for real reductive groups, inspired by the classfication works of Barbasch. In that setting, it was recently proved in \cite{DMB}. On the other hand, for unipotent representations of $p$-adic groups, the FPP bound was implicitly known and it is easy to establish, essentially due to the fact that, unlike the case of real reductive groups, the only possible reducibility hyperplanes for the principal series occur when a root equals $0$ or $1$. For a more general proof of FPP for $p$-adic groups see \cite{JLLMB}.

In the integral case, $\psi(\Fr)=1$, so $\psi$ is the same as a group homomorphism $\psi: SL(2,\bC)\times SL(2,\bC)\to G$. More concretely, this is equivalent to a pair of commuting Lie triples $\{e_L,h_L,f_L\}$ and $\{e_A,h_A,f_A\}$ such that
\[
h_L+h_A=2\lambda.
\]
The centraliser $G(s)=G(\lambda)$ is a Levi subgroup $M$ of $G$ and the $q$-eigenspace for $\Ad(s)$ in $\fg$ is the same the $(+1)$-eigenspace $V=\fg_1$ of $\ad(\lambda)$ in $\fg$. As explained in section \ref{s:2}, the $M$-space $V$ can be identified with $V=\fn/[\fn,\fn]$, where $\fn$ is the nilradical of the Jacobson-Morozov parabolic for $\lambda$. The commuting Lie triples have the property that
\[
e_L,e_A\in \fg_1=V,\ f_L,f_A\in \fg_{-1}=V^*,\ h_L,h_A\in\fg_0=\fk m.
\]
We may arrange that $h_L,h_A\in \ft$.

The Kazhdan-Lusztig parameter for $\varphi_\psi$ is the $M$-orbit $\CO_{\varphi_\psi}$ of $e_L$ in $V$. If $\CO$ is any $M$-orbit in $\fg_1$, there exists an \emph{adapted} Lie triple $\{e,h,f\}$, $e\in\CO$, $h\in\fg_0$, $f\in \fg_{-1}$, so the question is if we can find an appropriate commuting partner 
Lie triple.
\begin{definition}\label{d:pair}
An $M$-orbit $\CO$ is said to be \emph{of Arthur type} if it has an adapted Lie triple $\{e,h,f\}$ such that there exists another adapted triple $\{e',h,',f'\}$ that commutes with it and $h+h'=2\lambda$. In this case we say that $(\CO,\CO')$ is an \emph{Arthur pair}, where $\CO'=M\cdot e'$.
\end{definition}

The above expectation can be reformulated more precisely as follows.

\begin{conjecture}[Integral unitary representations] \label{c:main}Suppose $\varphi$ is a Langlands parameter with $I_F\subseteq\ker\varphi$ and $\varphi(\Fr)=q^\lambda$, $\lambda$ dominant integral in FPP. Let $P=MN$ be the parabolic subgroup defined by $\lambda$. The Levi subgroup $M=G(\lambda)$ acts on $V=\fg_1=\fn/[\fn,\fn]$ with finitely many orbits. Suppose $\pi(\C E)\in \Pi(\varphi)$ is parametrised by the simple $M$-equivariant perverse sheaf  $\C E$. Then $\pi(\C E)$ is unitary if and only if there exists an Arthur pair $(\CO,\CO')$ of $M$-orbits in $V$ (Definition \ref{d:pair}) such that $\overline {T^*_\CO V}$ occurs in the characteristic cycle $CC(\C E)$.

\end{conjecture}

\begin{example}
The irreducible spherical representation $\pi_0$ with Satake parameter $q^\lambda$ is parametrised by $\C E_0$ (the skyscraper sheaf at $0$) and $CC(\C E_0)=\{0\}\times V^*=\overline {T^*_0 V}$. If the zero orbit is of Arthur type, it must be the case that $\lambda=\frac 12 h$, for the neutral element $h$ in a Lie triple for the open dense orbit. The conjecture implies that $\pi_0$ is unitary if and only if $2\lambda$ is  the middle element of a Lie triple in $\fg$. This was verified in \cite{CH} using the classification of the spherical unitary dual \cite{Ba,BC,Ci-F4}.
\end{example}

Let $w_0$ be the longest element in the finite Weyl group of $G$. Since $w_0\cdot h=-h$ for every middle element of a Lie triple in $\fg$ (\cite{GS}), it immediately follows that for Arthur pairs to exist at all at $\lambda$, we must have $w_0\cdot \lambda=-\lambda$. Moreover, if $\{e,h,f\}$ and $\{e',h',f'\}$ are commuting Lie triples, $\{e+e',h+h',f+f'\}$ is also a Lie triple. This means that $2\lambda$ must be the middle element of a nilpotent orbit, in fact, an even nilpotent orbit since $\lambda$ is assumed to be integral. Hence a weaker form of Conjecture \ref{c:main} is:

\begin{conjecture}[Weaker form]
    If $\lambda$ is integral, there exist unitary representations with Langlands parameter $\varphi$ such that $\varphi(W_F)=q^{\langle\lambda\rangle}$ if and only if $\lambda$ is half the middle element of an even nilpotent $G$-orbit.
\end{conjecture}
The ``only if'' part is known by the work of Kazhdan-Lusztig (the Iwahori case) and Lusztig (in general): if $\lambda=\frac 12h$ for $h$ a middle element of a nilpotent $G$-orbit, then the open $G(\lambda)$-orbit in $\fg_1$ parametrises tempered, hence unitary, representations. 

\medskip

Since $M$ acts with finitely many orbits on $V$, there exists a bijection, the \emph{Piasetskii duality}, between the $M$-orbits in $V$ and the $M$-orbits in $V^*$. Denote the dual of $\CO$ by $\CO^*$. This is an $M$-orbit in $V^*=\fg_{-1}$. Let $\{e^*,h^*,f^*\}$ be a Lie triple for $\CO^*$, with $e^*\in\fg_1, h^*\in\fk t, f^*\in\fg_1$. Let $\CO^\vee$ be the $M$-orbit of $f^*$, this is an $M$-orbit in $V=\fg_1$. Set $e^\vee=f^*$, $h^\vee=-h^*$, and complete to a Lie triple $\{e^\vee,h^\vee,f^\vee\}$, $f^\vee\in \fg_{-1}$. This can be done by taking the image of $\{e^*,h^*,f^*\}$ under the Chevalley involution of $\fg$. The pair $(\CO,\CO^\vee)$ is a natural candidate for constructing Arthur pairs. Our main result is:

\begin{theorem}\label{t:main-abelian}
    Let $\lambda$ be an integral dominant infinitesimal character in FPP. Suppose the associated parabolic subalgebra $\fp=\fm+\fn$ has abelian nilradical $\fn$. In particular $M$ is a maximal Levi subgroup and $V=\fg_1=\fn$. 
    \begin{enumerate}
    \item If $w_0\lambda\neq -\lambda$, there are no Arthur pairs $(\CO,\CO')$, and no representations with infinitesimal character $q^\lambda$ are hermitian. 
        \item Suppose $w_0\lambda= -\lambda$. 
        For every $M$-orbit $\CO$, $(\CO,\CO^\vee)$ is an Arthur pair. 
        Every irreducible representation $\pi$ with infinitesimal character $q^\lambda$ is unitary.
    \end{enumerate}
\end{theorem}

The proof is in Propositions \ref{p:Arthur-pair} and \ref{p:alcove}. Notice that this verifies Conjecture \ref{c:main} in the abelian nilradical case, a particularly nice and uniform case. Claim (1) holds for all real $\lambda$, see Proposition \ref{p:hermitian}, not just the integral ones. 

\medskip

The classification of parabolic subgroups with abelian unipotent radical is recalled in Table \ref{ta:abelian}. It is worth noting that the condition $w_0\cdot \lambda=-\lambda$ is known to be equivalent to the prehomogeneous vector space $V$ being \emph{regular}, in the sense of the centraliser of a generic point being reductive (\cite{Rub}). In section \ref{s:explicit}, we discuss in detail the geometry of the $M$-space $\fn$ when $\fn$ is abelian and record the microlocal packets in Theorem \ref{t:abelian-micro}. This is essentially already available in different places in the literature, for example \cite{LW,LR,Pa,Roh,Sato,Ki}, and we include the results here for convenience.

The case of the abelian nilradical is part of a more general family of examples, namely when the $M$-space $V$ is irreducible and \emph{spherical}. As recalled in section \ref{s:2}, $V=\fn/[\fn,\fn]$ is irreducible if and only if $M$ is maximal. $V$ is spherical if a Borel subgroup of $M$ has an open orbit in $V$. The classification of irreducible spherical vector spaces is known from \cite{Kac}. An important feature is that for spherical vector spaces, the characteristic cycles are multiplicity free, see \cite{LW}. In a follow-up paper, we will consider the general maximal parabolic subgroups. 

\medskip

\noindent{\bf Acknowledgments.} The author thanks Hiraku Atobe and Alberto Minguez for a stimulating correspondence about the classification of the unitary dual at integral infinitesimal parameters.

\smallskip

\noindent{\it ``Say the line, Bart!''.} For the purpose of open access, the author has applied a CC BY public copyright licence to any author accepted manuscript arising from this submission.

\section{Setup}\label{s:2} Let $G$ be a complex connected almost simple group over $\mathbb C$. For the applications to representations of $p$-adic groups, $G$ is the Langlands dual group. Assume that $G$ has simply-connected derived subgroup. Fix a Borel subgroup $B$ and a maximal torus $T\subset B$. Let $\Delta\subset\Phi^+\subset\Phi$ be the sets of simple, positive, all roots, respectively, of $T$ in $G$. The positive roots $\Phi^+$ correspond to $B$. Let $\fg,\fk b,\ft$ be the corresponding Lie algebras. Let $\fg_\alpha$ denote the root space of $\alpha$ and fix a pinning $\{X_\alpha\in\fg_\alpha\mid\alpha\in \Phi^+\}$.

\begin{definition}\label{d:integral}
    An element $\lambda\in\ft$ is called \emph{integral} if $\langle\lambda,\alpha\rangle\in\bZ$ for all $\alpha\in\Phi$.
\end{definition}

Let $\{\omega^\vee_\alpha:\alpha\in\Delta\}$ be the set of fundamental coweights,i.e., the dual basis of $\Delta$ in $\ft$. Of course, $\lambda$ is integral if and only if $\lambda=\sum_{\alpha\in\Delta}c_\alpha\omega_\alpha^\vee$ with $c_\alpha\in\bZ$. Since we will be interested in the applications to the unitary representations of $p$-adic groups, we may restrict to those $\lambda$ that belong to the \emph{fundamental parallelipiped} (FPP)
\begin{equation}
    0\le \langle\lambda,\alpha\rangle\le 1,\text{ for all }\alpha\in\Delta.
\end{equation}

Equivalently, an integral $\lambda$ is in FPP if and only if $c_\alpha\in\{0,1\}$, for all $\alpha\in\Delta$. Fix an integral $\lambda\in\ft$ belonging to FPP.

Define $J\subseteq \Delta$ by $J=\{\alpha\in\Delta\mid c_\alpha=\langle\lambda,\alpha\rangle=0\}$. With this notation, $\lambda=\sum_{\alpha\in\Delta\setminus J}\omega_\alpha^\vee$. Let $\fm\supseteq \ft$ be the Levi subalgebra of $\fg$ defined by the simple roots $J$. Let $\fn=\bigoplus_{\beta\in\Phi^+,\langle\lambda,\beta\rangle>0} \fg_\beta$ and $\fp=\fm\oplus \fn$ the parabolic subalgebra of $\lambda$. Also denote $\fn^-=\bigoplus_{\beta\in\Phi^+,\langle\lambda,\beta\rangle<0} \fg_\beta$. Let $P=MN$ the corresponding parabolic subgroup. It is clear that, in this way, we have a one-to-one correspondence between the set of integral $\lambda$ in FPP and the set of standard parabolic subgroups $P$ with respect to $B$.

In this setting, the Kazhdan-Lusztig geometry involves the action of $M=Z_G(\lambda)$ on the $(+1)$-eigenspace of $\ad(\lambda)$ on $\fg$,
\[
\fg_1=\bigoplus_{\beta\in\Phi^+,\langle\lambda,\beta\rangle=1} \fg_\beta.
\]
Every positive root $\beta$ can be written uniquely as
\[
\beta=\sum_{\alpha'\in J} a_{\alpha'} \alpha'+\sum_{\alpha\in\Delta\setminus J} b_{\alpha}\alpha,
\]
$a_\alpha,b_{\alpha'}\in \bZ_{\ge 0}.$ 
Following \cite{Roh}, define $\mathsf{shape}(\beta)=\sum_{\alpha\in\Delta\setminus J} b_{\alpha}\alpha$ and $\mathsf{level}(\beta)=\sum_{\alpha\in\Delta\setminus J} b_{\alpha}.$
\begin{lemma}
    As an $M$-representation, $\fg_1\cong \fn/[\fn,\fn].$
\end{lemma}

\begin{proof}
    Since $\lambda=\sum_{\alpha\in\Delta\setminus J}\omega_\alpha^\vee$, we have $\langle \lambda,\beta\rangle=1$ if and only if $\sum_{\alpha\in \Delta\setminus J} b_\alpha \langle\lambda,\alpha\rangle=1$. Moreover, this is if and only if $\mathsf{level}(\beta)=1$ and $\mathsf{shape}(\beta)=\alpha$ for a root $\alpha\in\Delta\setminus J$. It means that $\fg_1=\bigoplus_{\mathsf{level}(\beta)=1} \fg_\beta$, and this is precisely $\fn/[\fn,\fn].$
\end{proof}
Denote $V=\fn/[\fn,\fn]$, a finite-dimensional $M$-representation. This is known as the \emph{first internal Chevalley module}. For every $\alpha\in\Delta\setminus J$, denote $V_\alpha=\text{span}\{X_\beta\mid \mathsf{shape}(\beta)=\alpha\}.$

\begin{proposition}[\cite{Roh}] 
\begin{enumerate}
    \item $V=\bigoplus_{\alpha\in\Delta\setminus J} V_\alpha$ is the decomposition of $V$ into irreducible $M$-subrepresentations. The irreducible $M$-module 
    $V_\alpha$ has highest weight $\max_\alpha$, the unique maximal-height positive root of shape $\alpha$.
    \item The dual (contragredient) $M$-representation is $V^*=\fn^-/[\fn^-,\fn^-]=\bigoplus_{\alpha\in\Delta\setminus J} V_{-\alpha}^*$. The irreducible $M$-module $V_\alpha^*$ has highest weight $-\alpha.$
\end{enumerate}
    
\end{proposition}

In particular, if $J$ is maximal and $\{\alpha\}=\Delta\setminus J$, then $V=V_\alpha$ is an irreducible $M$-representation. 

\

We look at the geometry of $M$ acting on the prehomogeneous space $V=\fn/[\fn,\fn]$. In the Kazhdan-Lusztig geometry, the relevant category is that of $M$-equivariant perverse sheaves on $V$. The simple objects in this category are in one-to-one correspondence with the pairs $(\CO,\C L)$, where $\CO$ is an $M$-orbit in $V$ and $\C L$ is an irreducile $M$-equivariant local system supported on $\CO$. These local systems are in one-to-one correspondence with the irreducible representations of the group of fundemantal group $\pi_1(\CO)\cong Z_M(x)/Z_M(x)^0$, for $x\in\CO$. 

By Riemann-Hilbert correspondence, there is an equivalence of categories between the category of $M$-equivariant perverse sheaves on $M$ and the category of (regular holonomic) $M$-equivariant $\C D$-modules on $V$ (\emph{$\C D_V$-modules}). The latter is equivalent with the category of modules for a finite quiver $\hat Q$, see \cite[Theorem 3.4]{LW}. There are two facts from \cite{LW} that will be useful in the next section:
\begin{enumerate}
    \item The characteristic cycles of all simple equivariant $\C D_V$-modules when $V$ is a spherical $M$-space are multiplicity free \cite[Corollary 3.19]{LW}.
    \item The category of $\C D_V$-modules has an involution $\widetilde\CF$ defined as the composition of the Fourier transform $\CF: \C D_V\text{-mod}\to  \C D_{V^*}\text{-mod}$ by the Chevalley automorphism. $\widetilde\CF$ induces a graph automorphism of $\hat Q$. For example, $\widetilde\CF$ interchages the delta $\C D$-module at $0$ with the module of functions of $V$. Moreover, if $S$ is a simple $\C D_V$-module with support $\overline\CO$, then $\widetilde\CF(S)$ has support $\overline{\CO^\vee}$. In the cases treated in the next section, these facts are sufficient to determine $\widetilde\CF$. Moreover, if
    \[
    CC(S)=\sum_i c_i [\overline{T_{\CO_i}V}]
    \]
    then
     \[
    CC(S)=\sum_i c_i [\overline{T_{\CO_i^\vee}V}].
    \]
\end{enumerate}
We mention that in \cite{evens-mirkovic}, it is shown that the action of the Aubert-Zelevinsky involution on the set of irreducible representations with Iwahori-fixed vectors corresponds to the action of $\widetilde\CF$ on the Kazhdan-Lusztig parameter space. 

\section{Parabolic subgroups with abelian unipotent radicals}\label{s:3}

A particularly nice example is when $\fn$ is abelian, so $V=\fn$. In this case $J$ is necessarily maximal, but not every maximal $J$ has this property. The classification of parabolic subgroups $P$  with abelian nilradical for $J=\Delta\setminus\{\alpha\} $ is given in Table \ref{ta:abelian}.

When $\fn$ is abelian, the structure of $M$-orbits on $V$ has a simple, uniform description. Denote $\Phi(1)=\{\beta\in \Phi^+\mid \langle \lambda,\beta\rangle=1\}$. 
Let $\gamma\in \Phi^+$ be the highest root. Notice that $\gamma\in \Phi(1)$. The \emph{upper canonical string}  $\gamma_1,\gamma_2,\dots,\gamma_r$ \cite{Pa} is obtained by setting $\gamma_1=\gamma$ and, inductively, defining $\gamma_{i+1}$ to be the highest root in $\Phi^i(1)=\{\beta\in\Phi(1)\mid  \beta \text{ orthogonal to }\gamma_1,\gamma_2,\dots,\gamma_i\}.$ The integer $r$ is the smallest integer such that $\Phi^{r+1}(1)=\emptyset.$

Set $\CO_0=\{0\}$ and let $\CO_i$ be the $M$-orbit with representative $x_i=\sum_{j=1}^i X_{\gamma_i}$, $i=1,\dots,r$. Also set $x_{r-i}^\vee=\sum_{j=i+1}^r X_{\gamma_i}$, $0\le i\le r-1$ and $x_0^\vee=0$. Recall from Introduction the duality $\CO\to\CO^\vee$ coming from Piasetskii duality.
 
\begin{theorem}[Orbit structure I {\cite[Theorem 2.1, Corollary 2.3, Lemma 3.1]{Pa}}]\label{t:abelian-class}
    \begin{enumerate}\ 
\item The set of $M$-orbits on $V$ is $\{0\}\cup \{\CO_i\mid 1\le i\le r\}$. In particular, there are $r+1$ orbits.
\item $\CO_i\subset\overline\CO_j$ if and only if $i\le j$.
\item $V$ is a spherical $M$-space.
\item 
$\CO_i^\vee=\CO_{r-i}$, $0\le i\le r$.
    \end{enumerate}
\end{theorem}

\begin{proof}
The first three claims are all in \cite{Pa}. For Claim (4), \cite[Lemma 3.1]{Pa} shows that the Piasetskii dual of $\CO_i$ is the $M$-orbit $\CO_{r-i}^*\subset V^*=\fn^-$ that has representative a sum of root vectors for $r-i$ orthogonal (negative) long roots in $\fn^-$. Then the claim follows by applying the Chevalley involution. 
\end{proof}

\begin{proposition}\label{p:Arthur-pair}
   Retain the notation as in Theorem \ref{t:abelian-class}. If $w_0\cdot \lambda=-\lambda$ then 
   each pair $(\CO_i,\CO_i^\vee)$ is an Arthur pair. If $w_0\cdot \lambda\neq -\lambda$, no $(\CO_i,\CO_i^\vee)$ is an Arthur pair.
\end{proposition}

\begin{proof}
    Let $\{X_{\gamma_i},H_{\gamma_i},X_{-\gamma_i}\}$ be the Lie triple for the root $\gamma_i$. Let $\CO_i$ be an orbit with the representative $x_i=\sum_{j=1}^i X_{\gamma_i}$ as above. Set $h_i=\sum_{j=1}^i H_{\gamma_i}\in \ft$. Since the roots $\gamma_i$, $1\le i\le r$ are orthogonal, $[h_i,x_i]=2x_i$. Complete to an adapted Lie triple $\{x_i,h,\bar x_i\}$, $\bar x_i=\sum_{j=1}^i X_{-\gamma_i}\in \fn^-$. This works because the roots $\gamma_i$, $\gamma_j$ are strongly orthogonal for all $i\neq j$, i.e., $\gamma_i\pm\gamma_j\notin \Phi\cup\{0\}$. Similarly construct the adapted Lie triple $\{x_{r-i}^\vee,h_{r-i}^\vee, {x_{r-i}'}^\vee\}$ for $\CO_i^\vee=\CO_{r-i}$, starting with $x_{r-i}^\vee=\sum_{j=i+1}^r X_{\gamma_i}$. The two Lie triples constructed in this way commute because of the strong orthogonality of the roots invoved.

    Finally, by \cite[Proposition]{Pa}, $\lambda=\frac 12(h_i+h_{r-i}^\vee)$ if and only if $w_0\cdot \lambda=-\lambda$.
\end{proof}

Another useful characterisation of the orbits is given in \cite{Roh}. Let $P=MN^-$ be the opposite parabolic subgroup. For every $w\in W$, let $K=J\cap w(J)$, and let $w'=w_J^0w_K^0$, where $W_J^0$, $w_K^0$ are the longest elements in $W_J$, $W_K$, respectively. Let $^JW^J$ denote the set of minimal length representatives of the double cosets $W_J\backslash W/W_J$.

\begin{theorem}[Orbit structure II {\cite[Theorem 2.10]{Roh}}]\label{t:abelian-class-2}
\ 
\begin{enumerate}
\item For $w\in W$, $P^-wP^-\cap N$ is a single $M$-orbit. In particular, there is a bijective correspondence between $M$-orbits on $N$ and $^JW^J$.
\item For $w\in~ ^JW^J$, $\dim (P^-wP^-\cap N)=\ell(w)+\ell(w')$.
\item For $x\in N$, $G\cdot x\cap N=M\cdot x$. In particular, each unipotent class in $G$ meets at most one $M$-orbit in $N$.
\end{enumerate}
\end{theorem}
We'll use these facts in the calculations in section \ref{s:explicit}.

\section{Unitarity}

In this section, we explain why all irreducible hermitian representations in Lusztig's unipotent category with (dominant) integral infinitesimal character $\lambda$, such that the corresponding parabolic subgroup has abelian unipotent radical are unitary. The key observation is that, in this situation, $\langle\lambda,\gamma\rangle=1$, where $\gamma$ is the highest positive root. This means that $\lambda$ is precisely the vertex of the fundamental alcove in the dominant Weyl chamber farthest from the origin. 

\begin{proposition}\label{p:alcove}
    If $0\le \langle\lambda',\gamma\rangle\le 1$, then all irreducible hermitian representations in Lusztig's unipotent category with infinitesimal character $\lambda'$ are unitary.
\end{proposition}

\begin{proof}
Firstly, suppose $\pi$ is an irreducible $\CG(F)$-representation with Iwahori-fixed vectors. Let $\chi_{\lambda'}$ denote the unramified character of the maximal torus $\C T(F)$ corresponding to the element $q^\lambda\in\ft$. Then $\pi$ is a subquotient of the unramified minimal principal series $X(\chi_{\lambda'})=i_{\C B(F)}^{\CG(F)}(\chi_{\lambda'})$ (normalised parabolic induction). The reducibility hyperplanes of an unramified principal series $X(\chi_\nu)$, $\nu\in \bR\Phi$ are exactly the hyperplanes $\langle\nu,\alpha\rangle=1$, $\alpha\in\Phi$. Since $X(\chi_0)$ is irreducible and unitary (being unitarily induced), a classical deformation argument (``complementary series'') implies that, if hermitian, $X(\chi_\nu)$ remains unitary as we deform $\nu$ from $0$ towards the first hyperplane of reducibility. If $\nu$ is chosen dominant, this implies that for all $\nu$ in the fundamental alcove $0\le\langle\nu,\gamma\rangle<1$, $X(\chi_\nu)$ is still unitary. In the limit, i.e., the boundary of the alcove, all the subquotients of $X(\chi_\nu)$ are also unitary. In particular, so are the hermitian subquotients of $X(\chi_\gamma)$.

\medskip

Next, let $\pi$ be an irreducible unipotent representation with infinitesimal character $\lambda'$, not necessarily with Iwahori vectors. By \cite{Lu-unip}, $\pi$ is in a Bernstein component of the category of smooth representations that is equivalent to the module category of an affine Hecke algebra with unequal parameters. The affine Hecke algebra is constructed from a Levi subgroup $L_0$ of $G$ that has a pair $(\bO_0,\C L_0)$ consisting of a nilpotent $L$-orbit $\bO_0$ and a cuspidal $L_0$-equivariant local system $\C L_0$ supported on $\bO_0$. Let $\CO_r$ be the unique open orbit of $G(\lambda')$ in $\fg_1(\lambda')$. A necessary condition is that $G\cdot \bO_0\subset \overline{G\cdot \CO_r}$. 

The classification of cuspidal $(\bO_0,\C L_0)$ is in \cite{Lu-cusp}. If $\lambda'$ is in the fundamental alcove, the nilpotent orbit $G\cdot \CO_r$ is small, and there are very few possibilities for non-toral $L_0$. By inspection, the only possibilities are:

\begin{enumerate}
    \item $G=Sp(2n)$, $L_0=Sp(2)\times GL(1)^{2(n-1)}$, $\bO_0$, the regular nilpotent orbit in $Sp(2)$ and $\C L_0$ the nontrivial local system;
    \item $G=SO(2n)$, $L_0=SO(4)\times GL(1)^{2(n-2)}$, $\bO_0$ the regular nilpotent orbit in $SO(4)$ and $\C L_0$ the nontrivial local system;
    \item $G=E_7$, $L_0=(3A_1)''$, $\bO_0$ the regular nilpotent orbit in $SL(2)^3$ and $\C L_0$ the local system nontrivial on each $SL(2)$ component.
\end{enumerate}

In (2), (3), the only representations that can occur in our case are supercuspidal, hence unitary. They correspond to $\lambda'=\lambda$, half the middle element of $\bO_0$ and $\CO_r\subset G\cdot \bO_0.$

In (1), the representations that occur are the subquotients of the intermediate principal series $i_{\C P(F)}^{\C G(F)}(\sigma\otimes \chi_{\nu})$ where $\C P\supset\C B$ is the parabolic subgroup for which the derived subgroup of its Levi subgroup is of type $B_1$,  $\sigma$ is a supercuspidal unipotent representation, and $\chi_\nu$ is an unramified character. The condition $0\le \langle\lambda',\gamma\rangle\le 1$ implies that $\nu$ is a continuous deformation from $0$ to at most the first reducibility hyperplane, so we are exactly in a similar situation to the Iwahori case above.

\end{proof}

\begin{proposition}\label{p:hermitian}
    An irreducible representation in Lusztig's unipotent category with infinitesimal character $\lambda\in \bR\Phi\subset \fk t$ is hermitian \emph{only if} $w_0\cdot \lambda=-\lambda.$ 
    
    If $w_0$ is central in $W$, then every irreducible representation with real infinitesimal character is hermitian.
\end{proposition}

\begin{proof}
    Let $\pi$ be an irreducible $\CG(F)$-representation with real infinitesimal character $\lambda$. As in the previous proof, there is a Levi subgroup $L_0\le G$ and a cuspidal datum $(\bO_0,\C L_0)$ in $L_0$, such that the Berstein component of the category of smooth representations to which $\pi$ is equivalent to the module category of an affine Hecke algebra $\CH$ (possibly with unequal parameters) constructed from $(\bO_0,\C L_0)$ in \cite{Lu-unip}. The algebra $\CH$ contains the finite Hecke algebra $\CH_q(W_0)$ with parameters for the finite Weyl group $W_0=N_G(L_0)/{L_0}$. The fact that this is a Weyl group is due to the special nature of the Levi subgroups $L_0$ supporting cuspidal data. It also contains an abelian subalgebra $\CA$ that can be identified with the ring of regular functions on a maximal torus in the centraliser in $G$ of a Lie triple for $\bO_0$, such that $\CH=\CH_q(W_0)\otimes_\bC \CA$ as a $(\CH_q(W_0),\CA)$-bimodule. See \cite{CMBO23} for a detailed explanation and a summary of Lusztig's construction. 

Let $\pi_0$ be the $\CH$-module corresponding to $\pi$. The algebra $\CH$ is isomorphic to a convolution algebra of functions on the $p$-adic group $\CG(F)$ and therefore comes with a natural $\star$-operation. It is known and easy that $\pi$ is $\CG(F)$-hermitian if and only if $\pi_0$ is $\CH$-hermitian with respect $\star$. See \cite{BC-equiv} for example.

The centre of $\CH$ is $\CA^{W_0}$ \cite{Lu-graded}. In the above equivalence, the central character of $\pi_0$ is the $W_0$-orbit of an element $\nu_0\in \fk t$ such that
\[
\lambda=\frac 12 h_0+\nu_0,
\]
for $h_0$ the middle element of a Lie triple for $\bO_0$. Because of the particular nature of $L_0$, $w(L_0)=L_0$ for all $w\in W$. In particular, $w_0\in N_G(L_0)$. Let $w_0^z$ denote the image of $w_0$ in $W_0$. Since $w_0$ maps $\Phi^+$ into $\Phi^-$, it follows that $w_0$ also maps all the positive roots of $\CH$ to negative, the former being vectors in $\bR\Phi^+$. Hence $w_0^z$ is the longest element of $W_0$.

Now $w_0\cdot \lambda=-\frac 12 h_0+w_0^z\cdot \nu_0$, and therefore
\[
w_0\cdot \lambda=-\lambda\text{ if and only if } w_0^z\cdot \nu_0=-\nu_0.
\]
In other words, we have reduced the question to one about modules over the affine Hecke algebra. Since $\CA$ is abelian and $\pi_0$ is finite dimensional, we can decompose $\pi_0$ into a sum of generalised $\CA$-eigenspaces. The corresponding set of eigenvalues $\Xi(\pi_0)$ are called the $\CA$-weight of $\pi_0$. By \cite{evens-mirkovic}, $\Xi(\pi_0)$ uniquely determines $\pi_0$ if $\pi_0$ has real central character. Let $\pi_0^h$ denote the hermitian dual of the $\CH$-module $\pi_0$. By \cite{BC-herm}, 
\[
\Xi(\pi_0^h)=-w_0^z(\Xi(\pi_0)).
\]
Hence $\pi_0^h\cong\pi_0$ if and only if $-w_0^z(\Xi(\pi_0))=\Xi(\pi_0)$. All weights of $\pi_0$ are $W_0$-conjugate to the central character $\nu_0$ which may be assumed $W_0$-dominant. This implies that if $\pi_0$ is hermitian then $-w_0^z\cdot\nu_0$ is $W_0$-conjugate to $\nu_0$, but since both are $W_0$-dominant, this is equivalent to $-w_0^z\nu_0=\nu_0$. 

For the partial converse, if $w_0$ is central in $W$, then it acts by the negative of the identity on $\fk t$, and in particular, so does $w_0^z$ on $\Xi(\pi_0)$. Hence, $\Xi(\pi_0^h)=\Xi(\pi_0)$, and therefore $\pi_0\cong \pi_0^h.$ We mention that this part was known (in the equivalent setting of graded affine Hecke algebras) from \cite[Lemma 3.1.1 and Corollary 4.3.2]{BC-herm}.
\end{proof}

To complete the proof of the unitarity part of Theorem \ref{t:main-abelian}, notice that by Propositions \ref{p:alcove} and \ref{p:hermitian}, it only remains to discuss the case when $w_0$ is not central, but $w_0\cdot\lambda=-\lambda$. Let $\lambda$ be an integral dominant infinitesimal character such that $P=MN$ has abelian radical $N$, as in the hypothesis of Theorem \ref{t:main-abelian}. By inspecting Table \ref{ta:abelian}, we see that the only cases remaining are:
\begin{enumerate}
    \item $G=GL(2n)$, $M=GL(n)\times GL(n)$, $V=\bC^n\otimes \bC^n$. As explained in section \ref{s:explicit}, there are $(n+1)$ orbits, all with connected isotropy groups. The corresponding irreducible representations of $\CG=GL(n,F)$  are the Langlands quotients of the parabolically induced representations
    \[
    i_{GL(1)^{n-k}\times GL(2)^{2k}\times GL(1)^{n-k}}^{GL(n,F)}(\mathsf{St}\otimes |\det|^{\underline \nu}),\quad \underline\nu=(\underbrace{\frac 12,\dots,\frac 12}_{n-k}, \underbrace{0,\dots,0}_{2k},\underbrace{-\frac 12,\dots,-\frac 12}_{n-k}),
    \]
    where $\mathsf{St}$ is the Steinberg representation of the Levi subgroup.
    \item $G=Spin(2n)$, $n$ odd, $M=GL(1)\times Spin(2n-2)$, $V=\bC\otimes \bC^{2n-2}$. There are three orbits, and only the open orbits has a disconnected isotropy group, with component group $\bZ/2\bZ$. The irreducible representations corresponding to the open orbit are tempered, hence unitary. The spherical representation is hermitian since $w_0\cdot\lambda=-\lambda$. Finally, the representation corresponding to the remaining orbit is the Langlands quotient of the parabolically induced representation
    \[
    i_{GL(2)\times GL(1)^{n-2}}^{PSO(2n,F)}(\mathsf{St}\otimes |\det|^{\underline\nu}),\quad \underline\nu=(\frac 12,\frac 12,\underbrace{0,\dots,0}_{n-2}).
    \]
\end{enumerate}
In both cases, the inducing Langlands datum $(L,\mathsf{St},\underline\nu)$, $L$ Levi subgroup, has the property that $w_0\cdot (L,\mathsf{St},\underline\nu)=(L,\mathsf{St},-\underline\nu)$. This implies that the Langlands quotient is hermitian, see for example \cite[Proposition 1.5]{BM-unit}. By Proposition \ref{p:alcove}, these representations are then unitary.

\section{Explicit results}\label{s:explicit}

We recall the explicit classification of parabolic subgroups with abelian unipotent radical, and give the orbit structure and the characteristic cycles in each case. Firstly, we fix the notation for the Dynkin diagrams.

\begin{enumerate} 
\item[(i)] If $G=GL(n)$, with Dynkin diagram
\[\xymatrix{1\ar@{-}[r] &2 \ar@{-}[r] &3  \ar@{-}[r] &\dotsb \ar@{-}[r] &(n-1)},
\]
every $\alpha\in \Delta$.
\item[(ii)] If $G=Sp(2n)$, with  diagram
\[\xymatrix{1\ar@{-}[r] &2 \ar@{-}[r] &3  \ar@{-}[r] &\dotsb \ar@{-}[r]&(n-1)\ar@{<=}[r] &n},
\]
only $\alpha_n$.
\item[(iii)] If $G=Spin(2n+1)$, with diagram
\[\xymatrix{1\ar@{-}[r] &2 \ar@{-}[r] &3  \ar@{-}[r] &\dotsb \ar@{-}[r] &n-1\ar@{=>}[r] &n},
\]
only $\alpha_1$.
\item[(iv)] If $G=Spin(2n)$, with diagram
\[\xymatrix{ &&&&& n-1\ar@{-}[ld]\\ &1 \ar@{-}[r]  &2  \ar@{-}[r] &\dotsb \ar@{-}[r] &n-2 \ar@{-}[rd] \\
&&&&&n}
\]
$\alpha_1,\alpha_{n-1},$ and $\alpha_n$.
\item[(vi)] If $G=E_6$, with diagram
\[\xymatrix{1\ar@{-}[r] &3 \ar@{-}[r]  &4\ar@{-}[d]  \ar@{-}[r] &5 \ar@{-}[r] &6,\\
&&2}
\]
$\alpha_1$ and $\alpha_6$.
\item[(vi)] If $G=E_7$, with diagram
\[\xymatrix{1\ar@{-}[r] &3 \ar@{-}[r]  &4\ar@{-}[d]  \ar@{-}[r] &5 \ar@{-}[r] &6 \ar@{-}[r] &7,\\
&&2}
\]
only $\alpha_7$.
\end{enumerate}

The corresponding irreducible $M$-representations $V$ are in Table \ref{ta:abelian}.

\begin{table}[ht]
    \centering
    \begin{tabular}{|c|c|c|c|c|}
    \hline
          $\Delta$ &Node &$M$ &$V$ &$|M\backslash V|$\\
          \hline
         $A_{n-1}$&$1\le \ell\le n-1$ &$GL(\ell)\times GL(n-\ell)$ &$\bC^\ell\otimes \bC^{n-\ell}$ &$\min(\ell,n-\ell)+1$\\
         \hline
         $C_n$ &$n$ &$GL(n)$ &$S^2\bC^n$ &$n+1$\\
         \hline
         $B_n$ &$1$ &$GL(1)\times Spin(2n-1)$ &$\bC\otimes \bC^{2n-1}$ &$3$\\
         \hline
         $D_n$ &$1$ &$GL(1)\times Spin(2n-2)$ &$\bC\otimes \bC^{2n-2}$ &$3$\\
         &$n-1$, $n$ &$GL(n)$ &${\bigwedge}^2 \bC^n$ &$\lfloor \frac n2\rfloor+1$\\
         \hline
         $E_6$ &$1$, $6$ &$GL(1)\times Spin(10)$ &$\bC\otimes S$, $S$ half-spin &$3$\\
         \hline
         $E_7$ &$7$ &$E_6\times GL(1)$ &$V(27)\otimes \bC$ &$4$\\
         \hline
         
    \end{tabular}
    \caption{$M$-representations on abelian $V=\fn$}
    \label{ta:abelian}
\end{table}

We now describe the explicit structure of the orbits $\CO_i$, $0\le i\le r$, their fundamental groups, the nilpotent $G$-orbits $\bO_i$ such that $\bO_i\cap\fn=\CO_i$, and the corresponding quivers $\hat Q$, the latter as in \cite{LW}.

\subsection{Type $A$} We can think of this case as the action of $M=GL(\ell)\times GL(k)$, $\ell+k=n$, on the space of matrices of size $\ell\times k$: $(A,B)\cdot X=AXB^{-1}$. Without loss of generality, suppose $\ell\ge k$. There are $k+1$ orbits, $\CO_0,\CO_1,\dots,\CO_k$, where $\CO_i$ consists of matrices of rank $i$. The corresponding nilpotent $G$-orbits are $(2^i,1^{n-2i})$, in the partition (Jordan form) notation. It is easy to see that all isotropy groups are connected. Thus there are $n+1$ simple $M$-equivariant $\C D_V$-modules labeled $(i)$  supported on $\CO_i$, $0\le i\le k$.

There are two cases, see \cite[Theorem 5.4]{LW}. If $k\neq \ell$, then the category of equivariant $\C D_V$-modules is semisimple. On the other hand, if $k=\ell$, the quiver $\hat Q$ is the connected graph
\begin{equation}
\begin{aligned}
\xymatrix{(0)\ar@<1ex>[r] &(1)\ar@<1ex>[l]\ar@<1ex>[r]&\dots\ar@<1ex>[l]\ar@<1ex>[r]&(k-1) \ar@<1ex>[l]\ar@<1ex>[r]&(k)\ar@<1ex>[l]},\\
\end{aligned}
\end{equation}
where all $2$-cycles are zero. Notice that the case $k=\ell$ is the only case when $w_0\cdot\lambda=-\lambda$.

\subsection{Type $C$} Here $M=GL(n)$ acting on $V=S^2\bC^n$. One can identify $V$ with the space of symmetric $n\times n$ matrices with the usual $GL(n)$-action: $A\cdot X=AXA^t$. There are $n+1$ orbits, given by the rank of the matrix. Let $\CO_i$ be the orbit of matrices of rank $i$, $0\le i\le n$. 

If we denote the simple roots of $C_n$ in the usual coordinates $\{\epsilon_1-\epsilon_2,\dots,\epsilon_{n-1}-\epsilon_n,2\epsilon_n\}$, in terms of the notation from Theorem \ref{t:abelian-class}, $\gamma_i=2\epsilon_i$, $1\le i\le n$, and the orbit $\CO_i$ has representative $\sum_{j=1}^i X_{2\epsilon_j}$.
Then $\dim \CO_i=\frac {i(2n+1-i)}2$. The corresponding nilpotent $G$-orbit is $\bO_i=(2^i,1^{2(n-i)})$ in partition notation.

To compute the isotropy groups, it is easier to work with matrices. Let $X_i=\left(\begin{matrix}I&0\\0&0\end{matrix}\right)$, $1\le i\le n$, be the representative of $\CO_i$ where $I$ is the $i\times i$ identity matrix. Then
\[
Z_{GL(n)}(X_i)=\{\left(\begin{matrix}A&C\\0&B\end{matrix}\right)\mid A\in O(i), B\in GL(n-i), C\in M_{i\times (n-i)}\}.
\]
In particular, $Z_{GL(n)}(X_i)^{\text{red}}=O(i)\times GL(n-i)$, and so $A_{GL(n)}(X_i)=\bZ/2$, for all $1\le i\le n$. For $i=0$, $X_0=0$, and the centralizer is $GL(n)$.

There are $2n+1$ simple $M$-equivariant $\C D_V$-modules labeled $(i)$ (for the trivial $\bZ/2$-character) and $(i)'$ (for the sign $\bZ/2$-character), supported on $\CO_i$, $1\le i\le n$, in addition to $(0)$. For consistency, set $(0)'=(0)$.

The quiver $\hat Q$ is computed in \cite[Theorem 5.9]{LW}. It has two connected components
\begin{equation}
\begin{aligned}
\xymatrix{(1-\epsilon)\ar@<1ex>[r] &(3-\epsilon)\ar@<1ex>[l]\ar@<1ex>[r]&\dots\ar@<1ex>[l]\ar@<1ex>[r]&(n-3) \ar@<1ex>[l]\ar@<1ex>[r]&(n-1)\ar@<1ex>[l]\ar@<1ex>[r]&(n)\ar@<1ex>[l]}\\
\xymatrix{(\epsilon)'\ar@<1ex>[r] &(\epsilon+2)'\ar@<1ex>[l]\ar@<1ex>[r]&\dots\ar@<1ex>[l]\ar@<1ex>[r]&(n-4)' \ar@<1ex>[l]\ar@<1ex>[r]&(n-2)'\ar@<1ex>[l]\ar@<1ex>[r]&(n)'\ar@<1ex>[l]},
\end{aligned}
\end{equation}
where all $2$-cycles are zero, and the other $n-1$ vertices are isolated.

\subsection{$Spin(m)$ on $\bC^m$} This covers the cases $B_n$ and $D_n$ with the first node, with $m=2n-1$ ($B_n$) and $m=2n-2$ ($D_n$). The action of $Spin$ factors through $SO$. There are three orbits, $\CO_0$, $\CO_1$, $\CO_2$, of dimensions $0$, $n-1$, $n$, respectively. The corresponding nilpotent $G$-orbits are $(1^{m+2})$, $(2^2,1^{m-2})$, and $(3,1^{m-1})$, respectively. In terms of the usual coordinates of type $B/D$, one can choose representatives $x_0=0$, $x_1=X_{\epsilon_1+\epsilon_2}$, and $x_2=X_{\epsilon_1+\epsilon_2}+X_{\epsilon_1-\epsilon_2}$, as in Theorem \ref{t:abelian-class}. 
The only orbit with a nontrivial component group is $\CO_2$, in which case this is $\bZ/2\bZ$.

There are four simple $M$-equivariant $\C D_V$-modules labeled $(0)$, $(1)$, $(2)$, and $(2)'$, same convention as in the previous subsection. By \cite[Theorem 5.16]{LW}, the quiver $\hat Q$ is as follows:

\begin{enumerate}
    \item If $m$ is even, then $(2)'$ is isolated and the other connected component is
\begin{equation}    
\begin{aligned}
\xymatrix{(0)\ar@<1ex>[r] &(1)\ar@<1ex>[l]\ar@<1ex>[r]&(2)\ar@<1ex>[l]}
\end{aligned}
\end{equation}
\item If $m$ is odd, there are two connected components
\begin{equation}    
\begin{aligned}
\xymatrix{(1)\ar@<1ex>[r] &(2)\ar@<1ex>[l]} &\qquad  \xymatrix{(0)\ar@<1ex>[r] &(2)'\ar@<1ex>[l]} 
\end{aligned}
\end{equation}
\end{enumerate}

\subsection{Type $D$} The case $M=GL(n)$ acting on $V={\bigwedge}^2 \bC^n$ can be interpreted as the action of $GL(n)$ on $n\times n$ skew-symmetric matrices by $A\cdot X=AXA^t$. There are $r+1$ orbits, where $r=\lfloor \frac n2\rfloor$, $\CO_0,\dots,\CO_r$, with $\CO_i$ consisting of the matrices of rank $2i$. Let $J_i$ denote an $i\times i$ invertible skew-symmetric matrix, and let $X_i=\left(\begin{matrix}J_i&0\\0&0\end{matrix}\right)$, $1\le i\le r$, be the representative of $\CO_i$. Then
\[
Z_{GL(n)}(X_i)=\{\left(\begin{matrix}A&C\\0&B\end{matrix}\right)\mid A\in Sp(2i), B\in GL(n-2i), C\in M_{i\times (n-2i)}\}.
\]
In particular, $Z_{GL(n)}(X_i)^{\text{red}}=Sp(2i)\times GL(n-2i)$ is connected for all $1\le i\le r$. 
Thus the isotropy groups are all connected. There are two cases, see \cite[Theorem 5.7]{LW}. If $n$ is odd, then the category of equivariant $\C D_V$-modules is semisimple. On the other hand, if $n$ is even, the quiver $\hat Q$ is the connected graph
\begin{equation}
\begin{aligned}
\xymatrix{(0)\ar@<1ex>[r] &(1)\ar@<1ex>[l]\ar@<1ex>[r]&\dots\ar@<1ex>[l]\ar@<1ex>[r]&(r-1) \ar@<1ex>[l]\ar@<1ex>[r]&(r)\ar@<1ex>[l]},\\
\end{aligned}
\end{equation}
where all $2$-cycles are zero. 

\subsection{$E_6$} Consider the case of root labeled $1$, the other being similar. This is the case $\bC^*\times Spin(10)$ acting on the even half-spin module $V$. There are three orbits $\CO_0$, $\CO_1$, $\CO_2$ of dimensions $0$, $11$, and $16$. The corresponding nilpotent $G$-orbits are $1$, $A_1$, and $2A_1$, respectively. In terms of Theorem \ref{t:abelian-class}, the corresponding representatives are
\[
x_0=0,\quad x_1=X_{\frac 12(1,1,1,1,1,-1,-1,1)},\quad x_2==X_{\frac 12(1,1,1,1,1,-1,-1,1)}+=X_{\frac 12(-1,-1,-1,-1,1,-1,-1,1)},
\]
where the roots are in the standard Bourbaki coordinates. Recall that $x_1$ is  the root vector for the highest root. 
In terms of Theorem \ref{t:abelian-class-2}, the corresponding Weyl group elements are
\[
1,\quad s_1, \quad w^0 w^0_J.
\]
If $w=s_1$, one sees that $K={2,4,5,6}$, so $\ell(w')=\ell(w_J^0)-\ell(w_K^0)=20-10=10$, giving 
\[
\dim\CO_1=11.
\]
If $w=w^0w^0_J$, $K=J$, so $w'=1$, and $\ell(w)=\ell(w^0)-\ell(w^0_J)=36-20=16$. Hence $\dim \CO_2=16$, as it should be since this is the open dense orbit.

The isotropy groups are all connected, see the argument in \cite[\S5.5, Case 5]{LW}, for example. This is based on the following

\begin{lemma}[{\cite[Lemma 4.13, Remark 4.10]{LW}}]\label{l:centraliser-criterion}

\begin{enumerate}
    \item If $M=H\times \bC^\times$, where $H$ is a connected reductive group and $v$ is a highest weight vector in the irreducible $H$-representation $V$, the the isotropy group of $v$ in $M$ is connected.
    \item If the complement of the open $M$-orbit $\CO$ in the $V$ has codimension $\ge 2$ (equivalently, if $V$ is not regular), then $\CO$ is simply connected.
\end{enumerate}
\end{lemma}

Hence, there are also three simple equivariant $\C D_V$-modules. By {\it loc. cit.}, the quiver $\hat Q$ consists of three isolated vertices.

\subsection{$E_7$} This is the case $E_6\times \bC^*$ acting on $V(27)$. The upper canonical string in Bourbaki coordinates is
\[
\gamma_1=-\epsilon_7+\epsilon_8,\quad \gamma_2=\epsilon_5+\epsilon_6,\quad \gamma_3=-\epsilon_5+\epsilon_6.
\]
leading to four orbits $\CO_0=\{0\}$, $\CO_1$, $\CO_2$, $\CO_3$, as in Theorem \ref{t:abelian-class}. The corresponding nilpotent $G$-orbits are $1$, $A_1$, $2A_1$, and $(3A_1)''$ in Carter's notation \cite{Car}. In terms of Theorem \ref{t:abelian-class-2}, the corresponding Weyl group elements are
\[
1,\quad s_7,\quad w_1, \quad  w^0 w^0_J,
\]
where $w_1=s_7s_6s_5s_4s_2s_3s_4s_5s_6s_7$. 
If $w=s_7$, then $K=\{1,2,3,4,5\}$ (type $D_5$), so $w'=w_J^0w_K^0$ has length $\ell(w')=36-20=16.$ Hence
\[
\dim\CO_1=17.
\]
If $w=w_1$, $\ell(w)=10$, then $K=\{2,3,4,5,6\}$ (type $D_5$), so $\ell(w')=16$ again. Hence
\[
\dim\CO_2=26.
\]
The isotropy group of $\CO_1$ is connected by Lemma \ref{l:centraliser-criterion}. It remains to decide the cases $\CO_2$ and $\CO_3$. These are not computed in \cite{LW}. We claim that $\CO_2$ is simply connected, but $\CO_3$ is not. By \cite{Lu-perverse}, see also \cite[Proposition 2]{Ci-multi} for a summary, if $\CO$ is an $M$-orbit in $\fg_1$, and $x\in \CO$, there exists a Levi subgroup $L$ of $G$ such that $x\in\fk l$ and $A_L(x)=A_M(x)$. Moreover, it is known and easy to see that the natural map $A_L(x)\to A_G(x)$ is injective. This means, in particular, that whenever $A_G(x)=\{1\}$, $A_M(x)$ must also be trivial. This is the case for all $3$ orbits $1$, $A_1$, $2A_1$, hence $A_M(x_2)=\{1\}$. Now, for $\CO_3$, by {\it loc. cit.}, $L=G$ (this is the \emph{rigid} case in the sense of \cite{Lu-perverse}), and so $A_M(x_3)=A_G(x_3)$. The nilpotent $(3A_1)''$ has component group $\bZ/2\bZ$ in simply-connected $E_7$, hence $A_M(x_3)=\bZ/2\bZ$.

There are five simple equivariant $\C D_V$-modules, labeled $(i)$ (with support $\CO_i$), $0\le i\le 3$, and $(3)'$. The corresponding quiver is \cite[\S5.5, Case 10]{LW}:

\begin{equation}
\begin{aligned}
\xymatrix{(0)\ar@<1ex>[r] &(1)\ar@<1ex>[l]\ar@<1ex>[r]&(2) \ar@<1ex>[l]\ar@<1ex>[r]&(3)\ar@<1ex>[l]},\\
\end{aligned}
\end{equation}
where all $2$-cycles are zero, and $(3)'$ is isolated.

\begin{theorem}\label{t:abelian-micro}
    Suppose $(M,V)$ is as in Table \ref{ta:abelian}. 
    \begin{enumerate}
        \item If all the orbits $\CO_0,\CO_1,\dots,\CO_r$ are simply connected, then $CC((i))=\overline{T^*_{\CO_i}V}$ and $\CA(\CO_i)=\{(i)\}$, $0\le i\le r$. This is the case when 
        \begin{enumerate}
            \item $\Delta=A_{n-1}$ and any $1\le \ell\le n$;
             \item $\Delta=D_n$ and $\ell=n-1,n$;
             \item $\Delta=E_6$ and $\ell=1,6$;
        \end{enumerate}
        \item If $\Delta=C_n$, with orbits $\CO_0,\CO_1,\dots,\CO_n$, then 
        \[
        CC((i))=\begin{cases}\overline{T^*_{\CO_i}V},& n\equiv i (\text{mod }2),\\
        \overline{T^*_{\CO_i}V}+\overline{T^*_{\CO_{i-1}}V},& n\not\equiv i (\text{mod }2)
        \end{cases}, \quad CC((i)')=\begin{cases}\overline{T^*_{\CO_i}V},& n\not\equiv i (\text{mod }2),\\
        \overline{T^*_{\CO_i}V}+\overline{T^*_{\CO_{i-1}}V},& n\equiv i (\text{mod }2)
        \end{cases}
        \]
    The microlocal packets are $\CA(\CO_n)=\{(n),(n)'\}$, $\CA(\CO_0)=\{(0),(1)\}$, and $\CA(\CO_i)=\{(i),(i)',(i+1)\}$ if $i\neq 0,n$.
    \item If $\Delta=B_n$, with orbits $\CO_0,\CO_1,\CO_2$, then 
    \[
    CC((2))= \overline{T^*_{\CO_2}V},\  CC((0))= \overline{T^*_{\CO_0}V},\  CC((1))= \overline{T^*_{\CO_0}V}+\overline{T^*_{\CO_1}V},\ CC((2)')= \overline{T^*_{\CO_2}V}+\overline{T^*_{\CO_1}V}.
    \]
    The microlocal packets are $\CA(\CO_2)=\{(2),(2)'\}$, $\CA(\CO_1)=\{(1),(2)'\}$, $\CA(\CO_0)=\{(0),(1)\}.$
    \item $\Delta=D_n$ and $\ell=1$, with orbits $\CO_0,\CO_1,\CO_2$, then 
    \[
    CC((2))= \overline{T^*_{\CO_2}V},\  CC((0))= \overline{T^*_{\CO_0}V},\  CC((1))= \overline{T^*_{\CO_1}V},\ CC((2)')= \overline{T^*_{\CO_2}V}+\overline{T^*_{\CO_1}V}+\overline{T^*_{\CO_0}V}.
    \]
    The microlocal packets are $\CA(\CO_i)=\{(i),(2)'\}$, $i=0,1,2$.

     \item $\Delta=E_7$ and $\ell=7$, with orbits $\CO_0,\CO_1,\CO_2,\CO_3$, then
      \[
    CC((i))= \overline{T^*_{\CO_i}V},\ i=0,1,2,3,  \ CC((3)')= \overline{T^*_{\CO_3}V}+\overline{T^*_{\CO_3}V}+\overline{T^*_{\CO_1}V}+\overline{T^*_{\CO_0}V}.
    \]
     The microlocal packets are $\CA(\CO_i)=\{(i),(2)'\}$, $i=0,1,2,3$.
\end{enumerate}

\end{theorem}

\begin{proof}
    The twisted Fourier transform $\widetilde{\C F}$ induces a graph auto of the quiver $\hat Q$. It can be determined explicitly in the above cases using the description of $\hat Q$ and the remarks at the end of section \ref{s:2}. 

In addition, the $\C D$-modules that correspond to cuspidal local systems on the open orbit, e.g., $(2)'$ in case (4) and $(3)'$ in case (5), are fixed by $\widetilde{\C F}$. Then a result of Ginzburg\footnote{The author thanks Lucas Mason-Brown for the reference.} \cite[Theorem 3.2]{Gi} implies that their characteristic varieties contain the conormal bundles to all the orbits.
    
    For the cases above, these ideas are sufficient to determine the characteristic cycles, but, in general that might not be the case.
See also \cite{LW,LR} and the references therein, where these results  appear as well. 

\end{proof}

\ifx\undefined\bysame
\newcommand{\bysame}{\leavevmode\hbox to3em{\hrulefill}\,}
\fi


\begin{thebibliography}{MMMM}


\bibitem[ABV]{ABV} J.~Adams, D.~Barbasch, D.~Vogan, \emph{The Langlands classification
and irreducible characters for real reductive groups}, Progress in Mathematics, vol.
{\bf 104}, Birkh\" auser Boston, Inc., Boston, MA, 1992

\bibitem[Ar]{Ar} J.~Arthur, \emph{The endoscopic classification of representations. Orthogonal and symplectic groups}, American
Mathematical Society Colloquium Publications, {\bf 61}. American Mathematical Society, Providence, RI, 2013.
xviii+590pp.

\bibitem[AM]{AM} H.~Atobe, A.~Minguez, \emph{Unitary dual of p-adic split SO(2n+1) and Sp(2n): The good parity case (and slightly beyond)}, preprint 2025, \texttt{arXiv:2505.09991}.

\bibitem[Au]{Au-inv} A.-M. Aubert, \emph{Dualit\' e dans le groupe de Grothendieck de la cat\' egorie des repr\' esentations lisses de longueur finie d'un groupe r\' eductif p-adique}, {Trans. Amer. Math. Soc.}, {\bf 347}(1995), no. 6,  2179--2189.

\bibitem[Ba]{Ba} D. Barbasch, \emph{The unitary spherical spectrum for split classical groups}, J. Inst. Math. Jussieu {\bf 9} (2010), no. 2, 265--356.

\bibitem[BC1]{BC} D. Barbasch, D. Ciubotaru, \emph{Whittaker unitary dual of affine graded Hecke algebras of type $E$}, Compos. Math. {\bf 145} (2009), no. 6, 1563--1616. 

\bibitem[BC2]{BC-equiv} D. Barbasch, D. Ciubotaru, \emph{Unitary equivalences for reductive $p$-adic groups},  {\it Amer. J. Math.} {\bf 135} (2013), no. 6, 1633--1674.

\bibitem[BC3]{BC-herm} D. Barbasch, D. Ciubotaru,  \emph{Hermitian forms for affine Hecke algebras}, \texttt{arXiv:1312.3316v2} (2015).

\bibitem[BM]{BM-unit} D. Barbasch, A. Moy, \emph{Unitary Spherical Spectrum for p-adic classical groups}, Acta Appl. Math. {\bf 44} (1996), 3--37.

\bibitem[BV]{BV} D. Barbasch, D.A. Vogan, Jr.,
\emph{Unipotent representations of complex semisimple groups}, Ann. Math. {\bf 121} (1985), issue 1, 41--110.

\bibitem[Car]{Car} R.W. Carter,
\emph{Finite Groups of Lie Type}, Wiley, Chichester (1985).


\bibitem[Ci1]{Ci-F4} D. Ciubotaru,  \emph{The $I$-spherical unitary dual of split $p$-adic $F_4$}, Represent. Theory {\bf 9} (2005) 94--137.

\bibitem[Ci2]{Ci-multi} D. Ciubotaru, \emph{Multiplicity matrices for the affine graded Hecke algebra}, J. Algebra {\bf 320} (2008) 3950--3983.

\bibitem[CH]{CH} D. Ciubotaru, M. Harris, \emph{On the generalized Ramanujan and Arthur conjectures over function fields}, to appear in {\it Annals Math.}


\bibitem[CMBO1]{CMBO23} D. Ciubotaru, L. Mason-Brown, E. Okada, \emph{Wavefront sets of unipotent representations of reductive p-adic groups II}, J. Reine Angew. Math. (Crelles Journal) {\bf 823} (2025), 191--253. 



\bibitem[CM]{CM} D. Collingwood, W. McGovern, \emph{Nilpotent orbits in semisimple Lie algebras}, Van Nostrand?Reinhold, New York (1993).


\bibitem[CFMMX]{Cun-et-al} C. Cunningham, A. Fiori, A. Moussaoui, J. Mracek, B. Xu, \emph{Arthur packets for p-adic groups by way of microlocal vanishing cycles of perverse sheaves, with examples}, Memoirs of the American Mathematical Society {\bf 276} (2022), No. 1353.

\bibitem[DMB]{DMB} D. Davis, L. Mason-Brown, \emph{The FPP Conjecture for Real Reductive Groups}, preprint 2024, \texttt{arXiv:2411.01372}.

\bibitem[EM]{evens-mirkovic} S. Evens, I. Mirkovi\' c, \emph{Fourier transform and the Iwahori-Matsumoto involution}, {Duke Math. J.} {\bf 86}(1997), no. 5, 435--464.

\bibitem[Gi]{Gi} V. Ginzburg, \emph{Characteristic varieties and vanishing cycles}, Invent. Math. {\bf 84} (1986), 327--402.

\bibitem[GS]{GS} P. Gunnels, E. Sommers, \emph{A characterization of Dynkin elements}, 	\texttt{arXiv:0212089v2} (2003).

\bibitem[HJLLZ]{Liu-et-al} A. Hazeltine, D. Jiang, B. Liu, C.-H. Lo and Q. Zhang, \emph{On representations of Arthur type and unitary
dual for classical groups}, Preprint 2024, \texttt{arXiv:2410.11806v1}.

\bibitem[JLLMB]{JLLMB}
D. Jiang, B. Liu, C.-H. Lo, L. Mason-Brown, 
\emph{The FPP Conjecture for p-adic Groups}, preprint 2025, {\texttt arXiv:2509.18320}.

\bibitem[Ka]{Kac} V. Kac, \emph{Some remarks on nilpotent orbits}, J. Algebra {\bf 64} (1) (1980), 190--213.

\bibitem[Ki]{Ki} T. Kimura, \emph{The $b$-functions and holonomy diagrams of irreducible regular prehomogeneous vector spaces}, Nagoya Math. J. {\bf 85} (1982), 1--80.



\bibitem[Lu1]{Lu-unip} G. Lusztig, \emph{Classification of unipotent representations of simple $p$-adic groups}, Int. Math. Res. Not.  (1995), no. 11, 517--589.

\bibitem[Lu2]{Lu-perverse} G. Lusztig, \emph{Study of perverse sheaves arising from graded Lie algebras}, Adv. Math. {\bf 112} (1995) 147--217.

\bibitem[Lu3]{Lu-cusp} G. Lusztig, \emph{Cuspidal local systems and graded Hecke algebras. I}, Publ. Math. Inst. Hautes \' Etudes Sci. {\bf 67} (1988), 145--202.

\bibitem[Lu4]{Lu-graded} G. Lusztig, \emph{Affine Hecke algebras and their graded version}, J. Amer. Math. Soc. {\bf 2}, 1989,
599--635.

\bibitem[LR]{LR} A.~L\"orincz, C. Raicu, \emph{Local Euler obstructions for determinantal varieties}, Topology and its Applications {\bf 313} (2022).

\bibitem[LW]{LW} A.~L\"orincz, U. Walther,
\emph{On categories of equivariant $\C D$-modules}, Adv. Math. {\bf 351} (2019), 429--478.



\bibitem[Pa]{Pa} D. Panyushev,
\emph{Parabolic Subgroups with Abelian Unipotent Radical as a Testing Site for Invariant Theory}, Canad. J. Math. {\bf 51} (3), 1999 616--635.

\bibitem[Ro]{Roh} G. Rohrle, \emph{On the structure of parabolic subgroups in algebraic groups}, J. Algebra {\bf 157} (1993), 80--115.

\bibitem[Ru]{Rub} H. Rubenthaler, \emph{Espaces pr\'ehomog\`enes de types parabolique}, Lect. Math. Kyoto Univ. {\bf 14} (1982), 189--221.

\bibitem[SK]{Sato} M. Sato, T. Kimura, \emph{A classification of irreducible prehomogeneous vector spaces and their relative invariants}, Nagoya Math. J. {\bf 65} (1977) 1--155.

\bibitem[Vo]{Vo-llc} D.~Vogan, \emph{The local Langlands conjecture}, Representation theory of groups and algebras, Contemp.
Math. {\bf 145}, American Mathematical Society, Providence (1993), 305--379.

\end{thebibliography}
\end{document}